\documentclass{amsart}
\usepackage{amssymb}
\usepackage{amscd}
\usepackage{verbatim}
\usepackage{epsfig}

\begin{document}
\newcommand{\pa}{\partial}
\newcommand{\CI}{C^\infty}
\newcommand{\dCI}{\dot C^\infty}
\newcommand{\Hom}{\operatorname{Hom}}
\newcommand{\supp}{\operatorname{supp}}
\newcommand{\Op}{\operatorname{Op}}
\renewcommand{\Box}{\square}
\newcommand{\ep}{\epsilon}
\newcommand{\Ell}{\operatorname{Ell}}
\newcommand{\WF}{\operatorname{WF}}
\newcommand{\WFb}{\operatorname{WF}_{\bl}}
\newcommand{\WFsc}{\operatorname{WF}_{\scl}}
\newcommand{\diag}{\mathrm{diag}}
\newcommand{\sign}{\operatorname{sign}}
\newcommand{\Ker}{\operatorname{Ker}}
\newcommand{\Ran}{\operatorname{Ran}}
\newcommand{\Span}{\operatorname{Span}}
\newcommand{\sX}{\mathsf{X}}
\newcommand{\sH}{\mathsf{H}}
\newcommand{\sC}{\mathsf{C}}
\newcommand{\sk}{\mathsf{k}}
\newcommand{\codim}{\operatorname{codim}}
\newcommand{\Id}{\operatorname{Id}}
\newcommand{\id}{\operatorname{id}}
\newcommand{\cl}{{\mathrm{cl}}}
\newcommand{\piece}{{\mathrm{piece}}}
\newcommand{\bl}{{\mathrm b}}
\newcommand{\scl}{{\mathrm{sc}}}
\newcommand{\cul}{{\mathrm{cu}}}
\newcommand{\ocul}{{\mathrm{1c}}}
\newcommand{\Psib}{\Psi_\bl}
\newcommand{\Psibc}{\Psi_{\mathrm{bc}}}
\newcommand{\Psibcc}{\Psi_{\mathrm{bcc}}}
\newcommand{\Psisc}{\Psi_\scl}
\newcommand{\Psiocu}{\Psi_\ocul}
\newcommand{\Psih}{\Psi_\semi}
\newcommand{\Psihh}{\Psi_{\semi,\cF}}
\newcommand{\Psisch}{\Psi_{\scl,\semi}}
\newcommand{\Psischh}{\Psi_{\scl,\semi,\cF}}
\newcommand{\Psiocuh}{\Psi_{\ocul,\semi}}
\newcommand{\Psiocuhh}{\Psi_{\ocul,\semi,\cF}}
\newcommand{\semi}{\hbar}
\newcommand{\Diff}{\mathrm{Diff}}
\newcommand{\Diffh}{\mathrm{Diff}_{\semi}}
\newcommand{\Diffhh}{\mathrm{Diff}_{\semi,\cF}}
\newcommand{\Diffsc}{\Diff_\scl}
\newcommand{\Diffsch}{\mathrm{Diff}_{\scl,\semi}}
\newcommand{\Diffschh}{\mathrm{Diff}_{\scl,\semi,\cF}}
\newcommand{\Diffocu}{\Diff_\ocul}
\newcommand{\Diffocuh}{\mathrm{Diff}_{\ocul,\semi}}
\newcommand{\Diffocuhh}{\mathrm{Diff}_{\ocul,\semi,\cF}}
\newcommand{\PP}{\mathbb{P}}
\newcommand{\BB}{\mathbb{B}}
\newcommand{\RR}{\mathbb{R}}
\newcommand{\Cx}{\mathbb{C}}
\newcommand{\NN}{\mathbb{N}}
\newcommand{\R}{\mathbb{R}}
\newcommand{\sphere}{\mathbb{S}}
\newcommand{\codimY}{k}
\newcommand{\dimX}{n}
\newcommand{\cO}{\mathcal O}
\newcommand{\cS}{\mathcal S}
\newcommand{\cP}{\mathcal P}
\newcommand{\cF}{\mathcal F}
\newcommand{\cL}{\mathcal L}
\newcommand{\cH}{\mathcal H}
\newcommand{\cG}{\mathcal G}
\newcommand{\cU}{\mathcal U}
\newcommand{\cM}{\mathcal M}
\newcommand{\cT}{\mathcal T}
\newcommand{\cX}{\mathcal X}
\newcommand{\cY}{\mathcal Y}
\newcommand{\cR}{\mathcal R}
\newcommand{\loc}{{\mathrm{loc}}}
\newcommand{\comp}{{\mathrm{comp}}}
\newcommand{\Tb}{{}^{\bl}T}
\newcommand{\Sb}{{}^{\bl}S}
\newcommand{\Nb}{{}^{\bl}N}
\newcommand{\Tsc}{{}^{\scl}T}
\newcommand{\Ssc}{{}^{\scl}S}
\newcommand{\Vf}{\mathcal V}
\newcommand{\Vfh}{{\mathcal V}_\semi}
\newcommand{\Vfhh}{{\mathcal V}_{\semi,\cF}}
\newcommand{\Vb}{{\mathcal V}_{\bl}}
\newcommand{\Vcu}{{\mathcal V}_{\cul}}
\newcommand{\Vsc}{{\mathcal V}_{\scl}}
\newcommand{\Vsch}{{\mathcal V}_{\scl,\semi}}
\newcommand{\Vschh}{{\mathcal V}_{\scl,\semi,\cF}}
\newcommand{\Vocu}{{\mathcal V}_{\ocul}}
\newcommand{\Vocuh}{{\mathcal V}_{\ocul,\semi}}
\newcommand{\Vocuhh}{{\mathcal V}_{\ocul,\semi,\cF}}
\newcommand{\Lambdasc}{{}^{\scl}\Lambda}
\newcommand{\etat}{\tilde\eta}
\newcommand{\xt}{\tilde x}
\newcommand{\scH}{{}^{\scl}H}
\newcommand{\Hh}{H_{\semi}}
\newcommand{\Hhh}{H_{\semi,\cF}}
\newcommand{\Hsc}{H_{\scl}}
\newcommand{\Hschh}{H_{\scl,\semi,\cF}}
\newcommand{\Hsch}{H_{\scl,\semi}}
\newcommand{\Hscloc}{H_{\scl,\loc}}
\newcommand{\Hscd}{\dot H_{\scl}}
\newcommand{\Hscb}{\bar H_{\scl}}
\newcommand{\Hocu}{H_{\ocul}}
\newcommand{\Hocuh}{H_{\ocul,\semi}}
\newcommand{\ff}{{\mathrm{ff}}}
\newcommand{\inter}{{\mathrm{int}}}
\newcommand{\Sym}{\mathrm{Sym}}
\newcommand{\be}[1]{\begin{equation}\label{#1}}
\newcommand{\ee}{\end{equation}}

\newcommand{\xisc}{\xi_{\scl}}
\newcommand{\etasc}{\eta_{\scl}}
\newcommand{\xiocu}{\xi_{\ocul}}
\newcommand{\etaocu}{\eta_{\ocul}}

\newcommand{\cutsphere}{\mathrm{PS}}

\newcommand{\level}{\mathsf{c}}
\newcommand{\foliation}{\mathsf{x}}
\newcommand{\loccoord}{y}
\newcommand{\Foliation}{\mathsf{X}}
\newcommand{\Loccoord}{\mathsf{Y}}

\newcommand{\bo}{\partial M} 
\newcommand{\zero}{^{(0)}}
\renewcommand{\r}[1]{(\ref{#1})} 
\newcommand{\mat}[4]{\left(\begin{array}{cc} #1 &#2\\#3 & #4 
\end{array}\right)}
\renewcommand{\d}{\mathrm{d}} 
\newcommand{\dsymm}{\mathrm{d}^{\mathrm{s}}}
\newcommand{\dsymmw}{\mathrm{d}^{\mathrm{s}}_\digamma}
\newcommand{\dsymmY}{\mathrm{d}^{\mathrm{s}}_Y}
\newcommand{\dsymmsc}{\mathrm{d}^{\mathrm{s}}_\scl}
\newcommand{\dsymmscw}{\mathrm{d}^{\mathrm{s}}_{\scl,\digamma}}

\setcounter{secnumdepth}{3}
\newtheorem{lemma}{Lemma}[section]
\newtheorem{prop}{Proposition}[section]
\newtheorem{proposition}{Proposition}[section]
\newtheorem{thm}{Theorem}[section]
\newtheorem{cor} {Corollary}[section]
\newtheorem{result}[lemma]{Result}
\newtheorem*{thm*}{Theorem}
\newtheorem*{prop*}{Proposition}
\newtheorem*{cor*}{Corollary}
\newtheorem*{conj*}{Conjecture}
\numberwithin{equation}{section}
\theoremstyle{remark}
\newtheorem{rem}{Remark}[section]
\newtheorem{remark} {Remark}[section]
\newtheorem*{rem*}{Remark}
\theoremstyle{definition}
\newtheorem{Def}{Definition}[section]
\newtheorem*{Def*}{Definition}

\newcommand{\mar}[1]{{\marginpar{\sffamily{\scriptsize #1}}}}
\newcommand\av[1]{}

\title[Semiclassical approach to geometric X-ray transforms]{A
  semiclassical approach to geometric X-ray transforms in the presence of
convexity}

\author[Andras Vasy]{Andr\'as Vasy}
\thanks{The author gratefully acknowledges support from the National
  Science Foundation under grant number  DMS-1664683 and DMS-1953987.}

\date{\today}

\address{Department of Mathematics, Stanford University, Stanford, CA
94305-2125, U.S.A.}
\email{andras@math.stanford.edu}
\subjclass{53C65, 35R30, 35S05}

\begin{abstract}
In this short paper we introduce a variant of the approach to
inverting the X-ray transform that originated in the author's work
with Uhlmann. The new method is based on semiclassical analysis and
eliminates the need for using sufficiently small domains and layer
stripping for obtaining the injectivity and stability results,
assuming natural geometric conditions are satisfied.
\end{abstract} 

\maketitle

\section{Introduction}

In this short paper we introduce a variant of the approach to
inverting the X-ray transform that originated in the author's work
with Uhlmann \cite{Uhlmann-Vasy:X-ray}. Here recall that on a compact
Riemannian manifold with boundary $M$ the X-ray
transform is the map $I$ that assigns to each $f\in\CI(M)$ the
function $If$ on $SM$, the unit
sphere bundle, defined by
$$
(If)(\beta)=\int_{\gamma_\beta} f(\gamma_\beta(t))\,dt,
$$
where $\gamma_\beta$ is the geodesic whose lift to $SM$ goes through $\beta$. (Other
similar families of curves work equally well, as observed by H.~Zhou
in the appendix to \cite{Uhlmann-Vasy:X-ray}. There is also no need to
consider $SM$ the {\em unit} sphere bundle; indeed it is convenient to
consider $If$ defined on $TM\setminus o$ as a homogeneous function of
degree $-1$. Also, compactness can be relaxed.) Here the geodesics are
assumed to be sufficiently well-behaved so that the integrals are over
finite intervals, i.e.\ the geodesics reach the boundary in finite
affine parameter; the more strict requirements later on make this
automatic. It is also useful to consider $M$ as a smooth domain in a manifold
with boundary $\tilde M$; in this case we regard $f$ as a function
supported in $M$ (via extension by $0$). The inverse problem is
to recover $f$ from $If$, i.e.\ to construct a left inverse, or at
least show that $I$ is injective with suitable stability
estimates. Typically one approaches such problems by considering the
normal operator $I^*I$, or some modification. In the present context
(using the above over-parameterization, in that many $\beta$
correspond to the same geodesic) $I^*$ is replaced by a closely
related operator of
the form
$$
(Lv)(z)= \int_{S_z M} v(\beta)\,d\nu(\beta)
$$
of integration along the geodesics through $z$; ideally one would like
$LI$ invertible, at least up to `trivial' errors.

The key idea of \cite{Uhlmann-Vasy:X-ray} was to introduce an
{\em artificial boundary}, which is a hypersurface in $M$, which meant that rather than working on all
of $M$, one initially attempts to recover $f$ from $If$ in the region
on one side of this hypersurface. More precisely, one is working with
a family of hypersurfaces which are level sets of a function $\xt$. This
function $\xt$ is required to be strictly concave from the side of the
super-level sets, i.e.\ $\frac{d(\xt\circ\gamma_\beta)}{dt}(t_0)=0$ implies $\frac{d^2(\xt\circ\gamma_\beta)}{dt^2}(t_0)>0$. If $\xt$ is normalized so that $M$ is contained in
$\xt\leq 0$, then the main result of \cite{Uhlmann-Vasy:X-ray} was
invertibility, in the above sense, in a region $\xt\geq -c$, where $c>0$
was sufficiently small, depending on various analytic quantities. Here
$\xt=-c$ is the artificial boundary, and the strict concavity is in
fact only
required in $\xt\geq -c$. This
could be repeated in a layer stripping argument, allowing a global
result after a multi-step process. The analytic heart of this
artificial boundary argument involves Melrose's {\em algebra of scattering
pseudodifferential operators} \cite{RBMSpec}; while the artificial
boundary is at a geometrically finite place, the analytic way of
obtaining a modified `normal operator' effectively pushes it to
infinity. This is done by, most crucially (another modification is
also needed), inserting a localizer $\chi(\beta)$ into the
formula for $L$ that concentrates on geodesics almost tangential to the level
sets of $\xt$, with the approximate tangency becoming more strict as
$\xt$ approaches the artificial boundary, i.e.\ as $x=\xt+c\to 0+$; the precise way this happens
determines the analytic structure:
$$
(Lv)(z)= \int_{S_z M} \chi(\beta)v(\beta)\,d\nu(\beta).
$$

The new method introduced in this paper is based on {\em semiclassical analysis} and
eliminates the need for using sufficiently small domains and layer
stripping for obtaining the injectivity and stability results,
assuming natural geometric conditions are satisfied. It
applies both to the localized problems (in the sense of the artificial
boundary), in which case it still uses
Melrose's scattering algebra \cite{RBMSpec}, and to global problems, in
which case one uses a variant of the standard semiclassical
algebra. As injectivity or stability statements, the results are the same as one would obtain with the original
techniques of \cite{Uhlmann-Vasy:X-ray}, but with a more transparent and streamlined proof. This is
reflected by the stronger technical theorems on the modified normal operator
for the X-ray transform. The analytic heart of the approach is again to introduce a
localizer $\chi_h$, which now also depends on $h$; for $h$ small
this again localizes very close to geodesics tangential to level sets
of $\xt$:
$$
(L_hv)(z)= \int_{S_z M} \chi_h(\beta)v(\beta)\,d\nu(\beta).
$$

Concretely, on manifolds of dimension $\geq 3$, in the case of no conjugate points but with a convex
foliation still, we directly obtain a
modified normal operator that is invertible; this involves the use of
a {\em semiclassical foliation pseudodifferential algebra} (the aforementioned variant), but not the scattering algebra, and
it also eliminates the need for making small steps (thin layers) in
the layer stripping approach. One in fact needs a weaker requirement
on the lack of conjugate points for curves from point of tangency to the foliation,
which in dimension $> 3$ can be further weakened similarly to the work of
Stefanov and Uhlmann \cite{Stefanov-Uhlmann:Integral}. Since in this
case there is no need to renormalize $\xt$ as there is no artificial
boundary, in order to simplify the notation we write $x=\xt$; this
allows a notationally uniform treatment later.

\begin{thm}\label{thm:main}
  Suppose $M$ is a compact Riemannian manifold with boundary of dimension $\geq 3$ equipped with a function $x$ with strictly
  convex level sets and $dx$ nonzero. Suppose also that
geodesics do not have points conjugate to their points of tangency to
the level sets of $x$. Then the semiclassically
  modified normal operator $A$, see \eqref{eq:mod-normal-op} with $\Phi(x)=-x$, of the geodesic X-ray transform is a {\em
    left invertible}
  elliptic order $-1$ pseudodifferential operator.
\end{thm}

Here left invertible means that there is
an order $1$ pseudodifferential operator $G$ on $\tilde M$ such that $GA=\Id$ on $\dot
H^s$ for all $s$ (this is the space of distributions in $H^s$ supported
in the domain, using H\"ormander's notation \cite{Hor}). An immediate
consequence, due to the fact that the operator $L$ (see
\eqref{eq:def-L}) used in the definition of our modified normal
operator $A$ is a standard Fourier integral operator for fixed
non-zero $h$, with appropriate order and canonical relation, is:

\begin{cor}
 Let $s\in\RR$. Under the hypotheses of the theorem,
the X-ray transform $I$ is injective on $\dot H^s$ and we have stability
estimates: there exists $C>0$ such that for all $f\in\dot H^s$, we
have $\|f\|_{\dot H^s(M)}\leq C\|If\|_{H^{s+1/2}(SM)}$.
\end{cor}

On the other hand, if conjugate points 
are present, one can still work in appropriately small layers as 
determined by the geometry (to eliminate the conjugate points), but 
without the need to further shrink the size of the steps to obtain 
invertibility as required in \cite{Uhlmann-Vasy:X-ray}; this approach still uses the scattering algebra at the 
artificial boundary.

\begin{thm}\label{thm:main-sc}
  Suppose $M$ is a compact Riemannian manifold with boundary of dimension $\geq 3$ equipped with a function $\xt$ with strictly
  concave level sets in $\xt\geq -c$, from the super-level sets, and $d\xt$
  nonzero. Suppose also that geodesics contained in the region $\xt\geq -c$
do not have points conjugate to their points of tangency to
the level sets of $\xt$. Then, with $x=\xt+c$, the semiclassically
  modified normal operator $A$, see \eqref{eq:mod-normal-op}, with
  $\Phi(x)=x^{-1}$ and with cutoff $\tilde\chi$ given in \eqref{eq:tilde-chi-sc}, of the geodesic X-ray transform is a {\em
    left invertible}
  elliptic order $(-1,-2)$ scattering pseudodifferential operator.
\end{thm}

The method is also applicable to other X-ray transform problems, such
as the X-ray transform on asymptotically conic, e.g.\ asymptotically Euclidean, spaces, studied in
work with Zachos, based in part on Zachos' work \cite{Zachos:Thesis}, where it is harder to implement the original `thin
layer' approach of \cite{Uhlmann-Vasy:X-ray}.

Notice that for nonlinear problems there is an additional role in
localizing near $\pa M$, not addressed by the semiclassicalization, namely if one uses a Stefanov-Uhlmann
pseudolinearization formula \cite{Stefanov-Uhlmann:Rigidity}, one
needs to make sure that the coefficients of the transform in the
formula are close to known values, typically at the boundary. This
need for localization can
be eliminated if the unknown quantity is a priori globally close to a given
background, in which case one can obtain injectivity results provided
one has injectivity results for the background, without having to
introduce the additional small semiclassical parameter, but in order
to obtain the prerequisite injectivity results for the background, the
semiclassical approach is still very useful.

The plan of this short paper is the following.
In Section~\ref{sec:semicl} we introduce the analytic ingredients,
namely the semiclassical foliation pseudodifferential algebras, and
then in Section~\ref{sec:X-ray} we use this for the analysis of the
X-ray transform. The whole of Section~\ref{sec:X-ray} consists of the
proofs of Theorems~\ref{thm:main} and \ref{thm:main-sc}.

I am very grateful to Richard Melrose, Plamen Stefanov and Gunther
Uhlmann for their interest in the project and for helpful comments.

\section{The semiclassical algebra}\label{sec:semicl}
In this section we discuss an inhomogeneous pseudodifferential semiclassical algebra
associated to a foliation $\cF$ on a manifold $M$. As is usual, it depends on a
semiclassical parameter, traditionally denoted by $h\in [0,1]$; it is
an $h$-dependent family of operators on $M$. The
standard semiclassical algebra is built from vector fields $hV$,
$V\in\Vf(M)$ (see e.g.\ \cite{Zworski:Semiclassical}), so semiclassical differential operators are in the
algebra over $\CI(M)$ generated by these, i.e.\ locally finite sums of
finite products of these with $\CI(M)$. Thus, in local coordinates
$z\in O\subset\RR^n$,
$P\in\Diffh^m(M)$ means that
$$
P=\sum_{|\alpha|\leq m} a_\alpha(z,h) (hD_z)^{\alpha},
$$
with $a_\alpha\in\CI(\RR^n\times[0,1])$, supported in $O$.
Thus, as a traditional differential operator, $P$ degenerates at
$h=0$, but as a semiclassical operator it does not; its semiclassical
principal symbol is
$$
p(z,\zeta)=\sum_{|\alpha|\leq m} a_\alpha(z,0) \zeta^{\alpha},
$$
obtained by replacing $hD_z$ by $\zeta$, and evaluating the
coefficients at $h=0$, and the operator is
semiclassically elliptic if there is $c>0$ such that
$$
|p(z,\zeta)|\geq c\langle\zeta\rangle^m,
$$
i.e.\ $p$ is non-vanishing and is elliptic in the standard sense.

Our new semiclassical foliation algebra is built from vector fields which are either
semiclassical in the sense above, i.e.\ $hV$, $V\in\Vf(M)$, or
$h^{1/2}$-semiclassical and tangent to the foliation:
$$
\Vfhh(M)=h\Vf(M)+h^{1/2}\Vf(M;\cF),
$$
where $\Vf(M;\cF)$ denotes the Lie algebra of vector fields tangent to
the foliation. In local coordinates, in which the foliation is locally
given
by $x=(x_1,\ldots,x_k)$ being constant, and remaining coordinates
(which are thus coordinates along the leaves) are
$y_1,\ldots,y_{n-k}$, this means that the semiclassical
foliation vector fields are
$$
\sum_{j=1}^k a_j(x,y,h) hD_{x_j}+\sum_{j=1}^{n-k} b_j(x,y,h) h^{1/2} D_{y_j}.
$$
Correspondingly, elements of the algebra of semiclassical foliation differential
operators of order $m$, $\Diffhh(M)$, are of the form
$$
P=\sum_{|\alpha|+|\beta|\leq m} a_{\alpha\beta}(x,y,h) (hD_x)^\alpha (h^{1/2}D_y)^\beta,
$$
and the semiclassical foliation principal symbol is
$$
p(x,y,\xi,\eta)=\sum_{|\alpha|+|\beta|\leq m} a_{\alpha\beta}(x,y,0) \xi^\alpha\eta^\beta,
$$
with semiclassical ellipticity meaning that there is $c>0$ such that
$$
|p(x,y,\xi,\eta)|\geq c\langle (\xi,\eta)\rangle^m.
$$

A somewhat different perspective of $\Vfhh(M)$ is that it is a
conformal version of the adiabatic Lie algebra. In the latter, with
$\ep$ the adiabatic parameter, one considers a fibration (rather than
just a foliation) $\cF$, and the sum of $\ep\Vf(M)$ and
$\Vf(M;\cF)$, so in local coordinates as above the vector fields are
$$
\sum_{j=1}^k a_j(x,y,\ep) \ep D_{x_j}+\sum_{j=1}^{n-k} b_j(x,y,\ep) D_{y_j}.
$$
Thus, with $\ep=h^{1/2}$, our Lie algebra is $\ep$ times the adiabatic
Lie algebra, i.e.\ is a conformal, more precisely, a $1$-conformal (in
that the conformal factor is the first power of the adiabatic
parameter $\ep$) version of the adiabatic Lie algebra. The conformal
factor makes this algebra more localized, just like in the comparison
of the more microlocalized scattering \cite{RBMSpec} and the more global b-algebras
of Melrose \cite{Melrose:Transformation, Melrose:Atiyah}; this is what allows for relaxing the requirements on $\cF$
to being a foliation.

Returning to our semiclassical perspective, we turn this into a
pseudodifferential operator algebra $\Psihh(M;\cF)$ as follows. First
starting locally with $\RR^n$ and the foliation $\cF_{\RR^n}$ given by the joint
level sets of the $x_j$, we consider symbols
$a$ with
$$
|(D_z^\alpha D_\zeta^\beta a)(z,\zeta,h)|\leq C_{\alpha\beta}\langle \zeta\rangle^{m-|\beta|},
$$
i.e.\ the standard semiclassical class (one can also require
differentiability in $h$; since this is a parameter, i.e.\ there is no
differentiation in it, the choice is mostly irrelevant), but
quantizing it according to the foliation as
\begin{equation*}\begin{aligned}\label{eq:semicl-foliation-quantize}
    &(A_h u)(x,y)=(Au)(x,y,h)\\
    &=(2\pi)^{-n} h^{-n/2-k/2}\int
e^{i(x-x')\cdot\xi/h+i(y-y')\cdot\eta/h^{1/2}} a(x,y,\xi,\eta,h)
u(x',y',h)\,d\xi\,d\eta\,dx'\,dy'.
\end{aligned}\end{equation*}
This gives a class $\Psihh(\RR^n;\cF_{\RR^n})$, where the uniformity
statement of the symbolic estimates
in $z$ on $\RR^n$ is not shown in the notation.

Of course, one can change variables as $\tilde\eta=h^{1/2}\eta$, to
obtain the usual semiclassical quantization
$$
(2\pi)^{-n} h^{-n}\int
e^{i(x-x')\cdot\xi/h+i(y-y')\cdot\tilde\eta/h} a(x,y,\xi,h^{-1/2}\tilde\eta,h)
u(x',y',h)\,d\xi\,d\tilde\eta\,dx'\,dy'
$$
of the symbol
$$
\tilde a(x,y,\xi,\tilde\eta,h)=a(x,y,\xi, h^{-1/2}\tilde\eta,h);
$$
regarded as a semiclassical symbol, $\tilde a$ is of a weaker type:
$$
|D_x^\alpha D_y^\beta D_\xi^\gamma D_{\tilde\eta}^\delta\tilde
a(x,y,\xi,\tilde\eta,h)|\leq C_{\alpha\beta\gamma\delta}h^{-|\delta|/2}\langle (\xi,h^{-1/2}\tilde\eta)\rangle^{m-|\gamma|-|\delta|},
$$
and while this is sufficient to push the semiclassical algebra
through, it is more precise to consider the foliation setup above.

\begin{rem}
As an aside, one can also consider this as a `blown-down' 2-microlocal
coisotropic
semiclassical algebra
corresponding to the coisotropic $\tilde\eta=0$, which corresponds
exactly to the joint characteristic set of the semiclassical vector
fields tangent to $\cF$. The type of this algebra is a $1/2$-type, in
that from the semiclassical perspective one blows up of the
coisotropic parabolically at $h=0$ (the parabolic direction being tangent to $h=0$), so the
scaling is $h^{-1/2}\tilde\eta$. More singular (with homogeneity $1$)
and thus delicate coisotropic algebras,
corresponding to hypersurfaces,
were introduced by Sj\"ostrand and Zworski
\cite{Sjostrand-Zworski:Fractal}. (Though it was used for a different
purpose and from a different perspective, the work
\cite{Gannot-Wunsch:Conormal} of Gannot and Wunsch introduced semiclassical
paired Lagrangian distributions to extend the work of de Hoop, Uhlmann
and Vasy \cite{DUV:Diffraction} from the non-semiclassical setting, and this relates closely
to 1-homogeneous 2-microlocalization at a coisotropic.) The `blown-down' adjective refers to
the fact that from this blow-up perspective, the standard
semiclassical behavior (i.e.\ what happens away from $\tilde\eta=0$)
is blown down, since $\tilde\eta\neq 0$ corresponds to
$|\eta|\to\infty$ as $h\to 0$, and we have placed joint symbolic
demands on $a$ in $(\xi,\eta)$.
\end{rem}

We now return to a discussion of the basic properties of the
semiclassical foliation algebra.
One can more generally allow $a$ in \eqref{eq:semicl-foliation-quantize} to depend on $z'=(x',y')$ as
well. The standard left- and right-reduction arguments, removing the
$z'$, resp.\ $z$, dependence apply, see e.g.\ \cite{Melrose:Microlocal-notes, Vasy:Minicourse}, and give asymptotic expansions, so
for instance the right-reduced version of $a(z,z',\zeta)$ is
$$
b(z',\zeta',h)\sim \sum_{\alpha,\beta}
\frac{1}{\alpha!\beta!}(-h D_\xi)^\alpha
(-h^{1/2}D_\eta)^\beta\pa_x^\alpha\pa_y^\beta a|_{z=z'=(x',y')},
$$
while the left-reduced version is
$$
c(z,\zeta,h)\sim \sum_{\alpha,\beta}
\frac{1}{\alpha!\beta!}(h D_\xi)^\alpha
(h^{1/2}D_\eta)^\beta\pa_{x'}^\alpha\pa_{y'}^\beta a|_{z'=z=(x,y)}.
$$
This gives (again, see \cite{Melrose:Microlocal-notes, Vasy:Minicourse}) that the semiclassical foliation pseudodifferential
operators form a filtered $*$-algebra (with respect to the $L^2$-inner
product):
\begin{equation*}\begin{aligned}
    &A\in \Psihh^{m}(\RR^n;\cF_{\RR^n}),\ B\in \Psihh^{m'}(\RR^n;\cF_{\RR^n})\Rightarrow
    AB\in\Psihh^{m+m'}(\RR^n;\cF_{\RR^n}),\\
    &A\in\Psihh^m (\RR^n;\cF_{\RR^n})\Rightarrow A^*\in\Psihh^m (\RR^n;\cF_{\RR^n}),
 \end{aligned}\end{equation*}
and moreover that with
$\sigma_m(A)=[a]\in S^m/h^{1/2} S^{m-1}$,
so for smooth (in $h^{1/2}$) $a$,
$$
\sigma_m(A)(z,\zeta)|_{h=0}=a(z,\zeta,0),
$$
we have
$$
\sigma_{m+m'}(AB)=\sigma_m(A)\sigma_{m'}(B),\qquad\sigma_m(A^*)=\overline{\sigma_m(A)}.
$$

One also has the standard elliptic parametrix construction. One says
that $A$, and its principal symbol $a$, are elliptic if $a$ has an
inverse $b\in S^{-m}/h^{1/2}S^{-m-1}$ in the sense that $ab-1\in
h^{1/2} S^{-1}$; this is equivalent to the lower bound
$$
|a(z,\zeta,h)|\geq c\langle\zeta\rangle^m,\qquad c>0,
$$
for $h$ small (i.e.\ there exists $h_0>0$ such that the estimate holds
for $h<h_0$).
Then there is a parametrix $B\in \Psihh^{-m}(\RR^n;\cF_{\RR^n})$ such that $AB-I,BA-I\in h^\infty\Psihh^{-\infty}(\RR^n;\cF_{\RR^n})$.

Furthermore, $\Psihh (\RR^n;\cF_{\RR^n})$ is invariant under local diffeomorphisms
preserving the foliation
as is easily seen by the standard Kuranishi trick; this allows the
introduction of the class $\Psihh(M,\cF)$ on manifolds, which has all
the analogous properties to $\Psihh (\RR^n;\cF_{\RR^n})$ discussed above. In addition, with $\Hhh^s(M)$ the foliation semiclassical
Sobolev space, i.e.\ the standard Sobolev space $H^s(M)$ but with the
natural $h$-dependent family of norms, so for $s\geq 0$ integer, locally,
$$
\|u\|_{\Hhh^s}^2=\sum_{|\alpha|+|\beta|\leq s}\|(hD_x)^\alpha (h^{1/2}
D_y)^\beta u\|^2_{L^2},
$$
for negative integer $s$ by duality, in general by interpolation (or
via the foliation Fourier transform, or
via elliptic ps.d.o's),
$$
\Psihh^m(M;\cF)\subset\cL(\Hhh^s(M),\Hhh^{s-m}(M))
$$
uniformly in $h$.

A great advantage of the semiclassical algebra, which is maintained by
the foliation semiclassical algebra, is that the error of the elliptic
parametrix construction is $O(h^\infty)$, thus small for $h$ small, as
an element of $\cL(\Hhh^s(M),\Hhh^{s}(M))$, so e.g.\ $BA=I+E$ as the
output of the elliptic parametrix construction means
that there exists $h_0>0$ such that $(\Id+E)^{-1}$ exists  for $h<h_0$
(and differs from $\Id$ by an element of
$h^\infty\Psihh^{-\infty}(M,\cF)$), and thus $A$ actually has a left
inverse, and similarly it also has a right inverse.

As in \cite{Uhlmann-Vasy:X-ray}, we actually work on an ambient
manifold $\tilde M$ with $M$ a domain with smooth boundary in it, and
$A\in\Psihh^m(\tilde M;\cF)$ is elliptic on a neighborhood of
$M$. Then there exists $B\in\Psihh^{-m}(\tilde M;\cF)$ such that
$AB-I,BA-I\in \Psihh^0(\tilde M;\cF)$ are in fact in
$h^\infty\Psihh^{-\infty}(\tilde M;\cF_{\RR^n})$ when localize to a
sufficiently small neighborhood of $M$, i.e.\ for suitable cutoffs
$\psi$, identically $1$ in $M$,
$$
\psi(AB-I)\psi,\psi(BA-I)\psi\in
h^\infty\Psihh^{-\infty}(\tilde M,\cF).
$$
So in particular, with $E=BA-I$, for
$v$ supported in $M$, $\psi v=v$, so $\psi BA\psi=\psi^2+\psi E\psi$
shows that
$$
(\Id+\psi E\psi)v=\psi BAv.
$$
Now $\Id+\psi E\psi$ is invertible for sufficiently small $h$, so
$(\Id+\psi E\psi)^{-1}\psi B$ is a left inverse for $A$ on
distributions supported in $M$.

There is an immediate extension of this algebra to the scattering
setting of Melrose \cite{RBMSpec}; this algebra actually can be locally
reduced to a standard H\"ormander algebra, which in turn was studied earlier by
Parenti \cite{Parenti:Operatori} and Shubin \cite{Shubin:Pseudodifferential}. For simplicity, since this is the only relevant case for us,
we consider only the case of a codimension one foliation given by a
boundary defining function $x$. Recall that scattering vector fields $V\in\Vsc(M)$
on a manifold with boundary $M$
are of the form $xV'$, $V'$ is a b-vector field, i.e.\ a vector field
on $M$ tangent to $\pa M$, so in local coordinates, they are of the
form
$$
a_0(x,y)x^2D_x+\sum a_j(x,y)xD_{y_j}.
$$
As mentioned, we take our foliation to be
given by the level sets of $x$, so the foliation tangent sc-vector
fields are locally
$$
\sum a_j(x,y)xD_{y_j}.
$$
The semiclassical version of $\Vsc(M)$ is
simply $\Vsch(M)=h\Vsc(M)$ (for which pseudodifferential operators
were introduced by Vasy and Zworski \cite{Vasy-Zworski:Semiclassical},
but in the local, Euclidean, setting this has a much longer history); the semiclassical foliation version is
$$
\Vschh(M;\cF)=h\Vsc(M)+h^{1/2}\Vsc(M;\cF).
$$
Thus, the semiclassical foliation scattering differential operators
take the form
$$
\sum_{\alpha+|\beta|\leq m} a_{\alpha\beta}(x,y,h)(hx^2D_x)^{\alpha} (h^{1/2}xD_y)^{\beta}.
$$
The corresponding pseudodifferential operators $A\in\Psischh^{m,l}(M,\cF)$
again arise by a modified semiclassical quantization of standard
semiclassical symbols $a$, i.e.\ ones satisfying (conormal in $x$)
symbol estimates
$$
|(xD_x)^\alpha D_y^\beta D_\tau^\gamma D_\mu^\delta
a(x,y,\tau,\mu,h)|\leq C_{\alpha\beta\gamma\delta} \langle
(\tau,\mu)\rangle^m x^{-l},
$$
namely
\begin{equation*}\begin{aligned}
    &A_h u(x,y)=Au(x,y,h)\\
    &=(2\pi)^{-n} h^{-n/2-1/2}\int e^{i\Big(\frac{x-x'}{x^2}\frac{\tau}{h}+\frac{y-y'}{x}\frac{\mu}{h^{1/2}}\Big)}a(x,y,\tau,\mu,h)\,u(x',y')\,\frac{dx'\,dy'}{(x')^{n+1}}\,d\tau\,d\mu.
  \end{aligned}\end{equation*}
Thus, in $x>0$, these are just the standard semiclassical foliation
operators, in $h>0$ the standard scattering pseudodifferential
operators, with the combined behavior near $x=h=0$. In particular we
have an elliptic theory as in the semiclassical foliation setting: if
$A$ is elliptic, meaning
$$
|a(x,y,\tau,\mu,h)|\geq cx^{-l}\langle(\tau,\mu)\rangle^m,\qquad c>0,
$$
for $h$ sufficiently small,
then there is a parametrix $B\in\Psischh^{-m,-l}(M,\cF)$ with
$$
AB-\Id,BA-\Id\in h^\infty\Psischh^{-\infty,-\infty}(M,\cF),
$$
and there exists
$h_0>0$ such that for $h<h_0$, $A\in\cL(\Hschh^{s,r},\Hschh^{s-m,r-l})$
is invertible with uniform bounds. One can again proceed with
localizing the elliptic parametrix construction as above in case one has a smooth
domain $M$ in an ambient space $\tilde M$.

\section{Global X-ray transform}\label{sec:X-ray}
We now consider the inverse problem for the X-ray transform
$$
If(\beta)=\int_{\gamma_\beta} f(\gamma_\beta(t))\,dt,
$$
where for $\beta\in SM$, $\gamma_\beta$ is the geodesic through
$\beta$ (or in fact other similar families of curves work equally
well), i.e.\ $\beta=(\gamma_\beta(0),\dot\gamma_{\beta}(0))\in S_{\gamma_\beta(0)}M$ (with
the dot denoting $t$-derivatives) utilizing the notation of \cite{Uhlmann-Vasy:X-ray}.
We overall follow the approach of
\cite[Section~4-5]{Stefanov-Uhlmann-Vasy:Rigidity-Normal} via oscillatory
integrals, rather than the blow-up analysis of
\cite{Uhlmann-Vasy:X-ray}. Concretely, the approach of the non-semiclassical proof of Proposition~4.2 in
\cite{Stefanov-Uhlmann-Vasy:Rigidity-Normal} underlies most of the
local arguments near $t=0$ and as we follow these quite closely, we
will be relatively brief.

With $x$ the function giving the foliation, writing
$x(\gamma_\beta(t))=\gamma_\beta^{(1)}(t)$, the concavity hypothesis
is that
\begin{equation}\label{eq:concavity-hyp}
\dot\gamma^{(1)}_\beta(t)=0\Longrightarrow\ddot\gamma^{(1)}_\beta(t)>0.
\end{equation}
By compactness considerations this implies that there exist $\ep>0$
and $C_0>0$
such that
$$
|\dot\gamma^{(1)}_\beta(t)|\leq\ep\Longrightarrow\ddot\gamma^{(1)}_\beta(t)\geq
C_0.
$$
It is convenient to take advantage of this also holding in a
neighborhood $M'$ of $M$ in $\tilde M$.

\begin{rem}\label{rem:bdy-convex}
  We actually do not need to make any convexity assumptions on $\pa
M$. However, if it not strictly convex, we need to consider the
geodesic segments as those in $M'$, and some of these may intersect
$M$ in a number of segments. This is not an issue below since knowing
$If$ in the sense of integrals along geodesic segments in $M$, one
also obtains $I' f$, the integrals along geodesic segments in $M'$
when $f$ is supported in $M$. {\em We do not make this distinction explicit
below; thus $I$ actually refers to $I'$ from this point on.} (There is
no such issue if $\pa M$ is strictly convex and one chooses $\pa M'$
appropriately.)
\end{rem}

We write, relative to our convex foliation and some coordinates,
denoted by $y$, along the level sets,
$\beta=(x,y,\lambda,\omega)$, so we write tangent vectors as
$$
\lambda\,\pa_x+\omega\,\pa_y,
$$
and use $\gamma^{(1)}$ to denote the $x$
component of $\gamma$, and similarly $\gamma^{(2)}$ to denote the $y$
component of $\gamma$ to avoid confusion.
Then we have
\begin{equation}\begin{aligned}\label{eq:gamma-exp}
\gamma_{x,y,\lambda,\omega}(t)&=(\gamma^{(1)}_{x,y,\lambda,\omega}(t)
, \gamma^{(2)}_{x,y,\lambda,\omega}(t))\\
&=\big(x+\lambda t+\alpha(x,y,\lambda,\omega)
t^2+t^3\Gamma^{(1)}(x,y,\lambda,\omega,t),\\
&\qquad\qquad\qquad y+\omega t+t^2\Gamma^{(2)}(x,y,\lambda,\omega,t)\big)
\end{aligned}\end{equation}
with $\Gamma^{(1)},\Gamma^{(2)}$
smooth functions of $x,y,\lambda,\omega,t$, $\alpha$ a smooth function
of $x,y,\lambda,\omega$, and $\alpha(x,y,0,\omega)\geq C>0$ by the
concavity from the super-level sets hypothesis; see
\cite[Section~3]{Uhlmann-Vasy:X-ray} and \cite[Proof of
Proposition~4.2]{Stefanov-Uhlmann-Vasy:Rigidity-Normal}. For us the
relevant regime will be $\lambda$ small; we shall restrict to an
arbitrarily small neighborhood of $\lambda=0$ via the semiclassical
localization.

This implies the following bound:

\begin{lemma}\label{lemma:gamma-1-est}
  There exists $T>0$ such that every geodesic reaches $\pa M'$ (thus
  $\pa M$) in
  affine parameter $\leq T$.

  Moreover, there exist $\lambda_0>0$ and $C>0$ such that for all
  $\beta=(x,y,\lambda,\omega)$ with $|\lambda|<\lambda_0$ and
  for $t$ in the closed interval on which $\gamma_\beta$ is defined, 
\begin{equation}\label{eq:gamma-1-est}
\gamma^{(1)}_{x,y,\lambda,\omega}(t)\geq x+\lambda t+Ct^2/2.
\end{equation}
\end{lemma}

\begin{proof}
The concavity hypothesis implies that any critical point of $\gamma^{(1)}$ in $M'$ is a strict
local mimimum, and $\dot\gamma^{(1)}$ can only change sign once and do
so non-degenerately since immediately to the left of any zero of
$\dot\gamma^{(1)}$ it is negative, and immediately to the right it is
positive by the concavity. Thus, for any geodesic either the sign of $\dot\gamma^{(1)}$ is
constant (non-zero) or there is a unique point on it with
minimal $\gamma^{(1)}$ in $M'$ (so either in $(M')^\circ$ or on $\pa
M'$).

Moreover, in case the minimum of $\gamma^{(1)}$ is reached at some $t_0$, $\dot\gamma^{(1)}(t)$ has the same sign as $t-t_0$.  Indeed, this is
so for sufficiently small $|t-t_0|$ by either the concavity or by
the minimum being on $\pa M'$, and if it
vanished for some $t>t_0$ (with $t<t_0$ similar), taking the infimum $t_1>t_0$ of the values of $t$ at which
this happens one concludes that $\dot\gamma^{(1)}(t_1)=0$, and
$\dot\gamma^{(1)}(t)>0$ for $0<t<t_0$, which is a contradiction in view
of the concavity hypothesis.

In addition, by the uniform concavity estimate, if
$\dot\gamma^{(1)}(t_1)\geq\ep$, then $\dot\gamma^{(1)}(t)\geq\ep$ for
$t\geq t_1$, and similarly if $\dot\gamma^{(1)}(t_1)\leq-\ep$ then $\dot\gamma^{(1)}(t)\leq-\ep$ for
$t\leq t_1$. Note that by the uniform concavity estimate,
$|\dot\gamma^{(1)}(t)|\leq\ep$ can only hold for an affine parameter
interval $2\ep/C_0$; and if the minimum of $\gamma^{(1)}$ is reached
at $t_0$, then for $|t-t_0|\geq \ep/C_0$ one necessarily has $|\dot\gamma^{(1)}(t)|\geq\ep$.

Taking into account that
$M'$ is compact so $x$ is bounded, we conclude that there exists $T>0$ such that every geodesic $\gamma_\beta$ reaches
$\pa M'$ in both directions in affine parameter $\leq T$: if $|x|\leq
C_1$ on $M'$, say, then this holds with $T=2\ep/C_0+4C_1/\ep$.

Turning to \eqref{eq:gamma-1-est},
it suffices to prove this for $\lambda=0$, and then it follows for sufficiently small $\lambda$ by
compactness taking into account that
it holds near $t=0$ by \eqref{eq:gamma-exp}. For $\lambda=0$, we have
seen that $\dot\gamma^{(1)}(t)$ has the same sign as $t$. Then by compactness one obtains a
positive lower bound for $\dot\gamma^{(1)}$ on any compact subset of
$(0,\infty)$. Since we have the estimate \eqref{eq:gamma-1-est} for
sufficiently small $t$, say $0<t<\delta$, using the positive lower bound
$\dot\gamma^{(1)}(t)$ for $t\geq\delta/2$ proves \eqref{eq:gamma-1-est}
for $t\geq 0$
at the cost of reducing $C$; $t\leq 0$ is analogous.
  \end{proof}

\begin{rem}
We recall from \cite{Uhlmann-Vasy:X-ray} that we needed to work in a sufficiently
small region so that there are no geometric complications, thus there the
interval $[-T,T]$ of integration in $t$, i.e.\ $T$, is such that 
$\ddot\gamma^{(1)}(t)$ is uniformly bounded below by a positive constant in the
region over which we integrate, see
the discussion in \cite{Uhlmann-Vasy:X-ray} above Equation~(3.1), and then
further reduced in Equations~(3.3)-(3.4) so that the map sending $(x,y,\lambda,\omega,t)$ to
the lift of $(x,y,\gamma_{x,y,\lambda,\omega}(t))$ in the resolved
space $\tilde M^2$ with the diagonal being blown up, is a diffeomorphism in $t\geq
0$, as well as $t\leq 0$. In the present paper the appearance of no conjugate points assumptions occurs
in a closely related manner, when dealing with the stationary phase
expansion, though we use the weaker concavity condition
\eqref{eq:concavity-hyp}, so even geometrically we reduce the
conditions impose. In addition, the extra restrictions in
\cite{Uhlmann-Vasy:X-ray} that arise from making the smoothing (thus
`trivial') error removable disappear here.
\end{rem}

We now introduce the weight $\Phi$ for the exponential conjugation of
our normal operator. Below we consider weights $\Phi=\Phi(x)$ which are
decreasing functions of $x$; in the global context, $\Phi(x)=-x$ will
be used, in the scattering context (in which $x>0$) $\Phi(x)=x^{-1}$.
Now, for $\Phi(x)=-x$,
$$
\Phi(\gamma^{(1)}_{x,y,\lambda,\omega}(t))-\Phi(x)\leq-\lambda t-Ct^2/2\leq-\frac{C}{2}\Big(t+\frac{\lambda}{C}\Big)^2+\frac{\lambda^2}{2C}\leq\frac{\lambda^2}{2C}.
$$
Hence, with $\hat\lambda=\lambda/\sqrt{h}$, $\hat t=t/\sqrt{h}$, the
rescaling which plays a key role below,
$$
h^{-1}(\Phi(\gamma^{(1)}_{x,y,\lambda,\omega}(t))-\Phi(x))\leq
-\frac{C}{2}\Big(\hat t+\frac{\hat \lambda}{C}\Big)^2+\frac{\hat
  \lambda^2}{2C},
$$
so for $\hat\lambda$ in a fixed compact set (and $h$ sufficiently
small, as is always assumed),
\begin{equation}\label{eq:Gaussian-damping}
\exp(h^{-1}(\Phi(\gamma^{(1)}_{x,y,\lambda,\omega}(t))-\Phi(x)))
\end{equation}
is uniformly bounded above by a Gaussian in $\hat t$.

We consider now the operator $L$ defined
by
\begin{equation}\label{eq:def-L}
Lv(z)=\int\tilde\chi(x,y,\lambda/h^{1/2},\omega)v(\gamma_{x,y,\lambda,\omega})\,d\lambda\,d\omega,
\end{equation}
with $\tilde\chi$ having compact support in the third variable, thus
localizing to $\lambda\sim h^{1/2}$, and hence to geodesics almost
tangent to the level sets of $x$ at the base point $(x,y)$, for $h$ small. Further
we consider the operator
\begin{equation}\label{eq:mod-normal-op}
A_h=e^{-\Phi(x)/h}L_h Ie^{\Phi(x)/h},
\end{equation}
which is thus
given by
\begin{equation*}\begin{aligned}
A_h f(z)=\int e^{-\Phi(x(z))/h}
&e^{\Phi(x(\gamma_{z,\lambda,\omega}(t)))/h} \\
&\tilde\chi(z,\lambda/h^{1/2},\omega)
f(\gamma_{z,\lambda,\omega}(t))\,dt\,|d\nu|,
\end{aligned}\end{equation*}
{\em where $A_h$ is understood to apply only to $f$ with
  support in $M$, thus for which the $t$-integral is in a fixed finite
  interval, say $[-T,T]$},
where $|d\nu|$ is a smooth positive density in $(\lambda,\omega)$, such
as $|d\lambda\,d\omega|$.
The first step is to prove:

\begin{prop} $A_h\in h\Psihh^{-1}(\tilde M;\cF)$.
  \end{prop}

\begin{proof}
The operator $A_h$ is the left quantization of the (a priori
tempered distributional) symbol $a_{h}$ where $a_{h}$ is
the inverse Fourier transform in the second variable $z'$ of the integral kernel: if $K_{A_h}$ is the
Schwartz kernel of $A_h$, then in the sense of oscillatory integrals (or
directly if the order of $a$ is sufficiently low)
$$
K_{A_h}(z,z')=(2\pi)^{-n}h^{-n/2-1/2}\int e^{i(x-x')\xi/h+(y-y')\eta/h^{1/2}}a_{h}(x,y,\xi,\eta)\,d\xi\,\d\eta,
$$
i.e.\ $(2\pi)^{-n}$ times the semiclassical foliation Fourier transform in $(\xi,\eta)$ of
$$
(x,y,\xi,\eta)\mapsto e^{ix\cdot\xi/h+iy\cdot\eta/h^{1/2}}
a_{h}(z,\zeta),
$$
so taking the
semiclassical foliation inverse
Fourier transform in $(x',y')$ yields
$(2\pi)^{-n}a_{h}(x,y,\xi,\eta)e^{ix\cdot\xi/h+iy\cdot\eta/h^{1/2}}$, i.e.
\begin{equation}\label{eq:aj-in-terms-of-kernel}
a_{h}(z,\zeta)=(2\pi)^{n}e^{-ix\cdot\xi/h-iy\cdot\eta/h^{1/2}}(\cF_{h,\cF}^{-1})_{(x',y')\to(\xi,\eta)}K_{A_h}(x,y,x',y').
\end{equation}
Here we are using local coordinates; we comment below on the
considerations when $z$ and $z'$ are far apart and cannot be analyzed
in the same coordinate chart.

Now,
\begin{equation*}\begin{aligned}
K_{A_h}(x,y,x',y')&=\int e^{-\Phi(x)/h} e^{\Phi(x(\gamma_{z,\lambda,\omega}(t)))/h}\tilde\chi(z,\lambda/h^{1/2},\omega)\\
&\qquad\qquad\delta(z'-\gamma_{z,\lambda,\omega}(t))\,dt\,|d\nu|\\
&=
(2\pi)^{-n} h^{-n/2-1/2}
\int e^{-\Phi(x)/h} e^{\Phi(x(\gamma_{z,\lambda,\omega}(t)))/h}\tilde\chi(z,\lambda/h^{1/2},\omega)\\
&\qquad\qquad e^{-i\xi'\cdot(x'-\gamma^{(1)}_{z,\lambda,\omega}(t))/h}e^{-i\eta'\cdot(y'-\gamma^{(2)}_{z,\lambda,\omega}(t))/h^{1/2}}\,dt\,|d\nu|\,|d\xi'|\,|d\eta'|;
\end{aligned}\end{equation*}
as remarked above, the $t$ integral is actually over a fixed finite
interval, say $|t|<T$, or one may explicitly insert a compactly supported cutoff
in $t$ instead. (So the only non-compact domain of integration is in
$(\xi',\eta')$, corresponding to the Fourier transform.)
Thus, taking the semiclassical foliation inverse Fourier transform in $x',y'$ and evaluating at $\xi,\eta$
gives
\begin{equation}\begin{aligned}\label{eq:semicl-full-symbol}
a_{h}(x,y,\xi,\eta)=
\int e^{-\Phi(x)/h} &e^{\Phi(x(\gamma_{z,\lambda,\omega}(t)))/h}\tilde\chi(z,\lambda/h^{1/2},\omega)\\
&\qquad 
e^{i\xi\cdot(\gamma^{(1)}_{z,\lambda,\omega}(t)-x)/h}e^{i\eta\cdot(\gamma^{(2)}_{z,\lambda,\omega}(t)-y)/h^{1/2}}\,dt\,|d\nu|.
\end{aligned}\end{equation}
The proof of the proposition is completed by showing that the right
hand side is actually $h$ times a symbol of order $-1$.

Notice that technically we are using local coordinates in
\eqref{eq:semicl-full-symbol}. For the
semiclassical foliation pseudodifferential operators for symbolic statements we should be
considering $z,z'$ in the same chart as well as when $z$
and $z'$ are apart and cannot be analyzed in the same chart. In the
latter case
$|t|$ bounded below by a positive constant, and
we show that $K_{A_h}$ is smooth and $O(h^\infty)$. This is
implied by the semiclassically Fourier transformed, in $z'$,
expression being Schwartz, i.e.\ $a_h$ (and its derivatives) being
$O(h^\infty\langle(\xi,\eta)\rangle^{-\infty})$; for this we do not
need to explicitly consider a coordinate chart in $z=(x,y)$.
We prove this decay below by
stationary phase considerations in $(\lambda,\omega,t)$, meaning that
the phase is not actually stationary; in this case the oscillatory factor
$e^{-i(\xi x/h+\eta\cdot y/h^{1/2})}$ is actually irrelevant.

We first consider the $|t|$ small behavior, say $|t|<T_0$, so a single chart can be
used in the analysis ($z,z'$ are in the same chart).
We change the variables of integration to $\hat t=t/\sqrt{h}$, and
$\hat\lambda=\lambda/{\sqrt{h}}$, so the $\hat\lambda$ integral is in fact over
a fixed compact interval, but the $\hat t$ one is over $|\hat t|<T_0/\sqrt{h}$
which grows as $h\to 0$; in this process we obtain an additional
factor of $h$ from the change of density.
We deduce that the
phase is
\begin{equation*}\begin{aligned}
\xi\cdot&(\gamma^{(1)}_{z,\lambda,\omega}(t)-x)/h+\eta\cdot(\gamma^{(2)}_{z,\lambda,\omega}(t)-y)/h^{1/2}\\    
&=\xi(\hat\lambda\hat t+\alpha\hat t^2+h^{1/2}\hat
t^3\Gamma^{(1)}(x,y,h^{1/2}\hat\lambda,\omega,h^{1/2}\hat t))\\
&\qquad+\eta\cdot(\omega\hat t+h^{1/2}\hat t^2\Gamma^{(2)}(x,y,h^{1/2}\hat\lambda,\omega,h^{1/2}\hat t)),
\end{aligned}\end{equation*}
while the exponent of the exponential damping factor (which we regard as a Schwartz
function, part of the amplitude, when one regards $\hat t$ as a
variable on $\RR$) is
\begin{equation*}\begin{aligned}
-\Phi(x)/h&+\Phi(x(\gamma_{z,\lambda,\omega}(t)))/h\\
&=x/h-\gamma^{(1)}_{x,y,\lambda,\omega}(t)/h\\
&=
-h^{-1}(\lambda t+\alpha
t^2+t^3\Gamma^{(1)}(x,y,\lambda,\omega,t))\\
&=-(\hat\lambda\hat t+\alpha\hat
t^2+\hat t^3 h^{1/2}\hat\Gamma^{(1)}(x,y,h^{1/2}\hat\lambda,\omega,h^{1/2}\hat t)),
\end{aligned}\end{equation*}
with $\hat\Gamma^{(1)}$ a smooth function. Thus, after the rescaling
of $t$ and $\lambda$ to $\hat t$ and $\hat\lambda$, the integrand of
\eqref{eq:semicl-full-symbol} is a smooth function of all
variables. In view of the Gaussian decay of the exponential damping
factor around \eqref{eq:Gaussian-damping}, the lack of compactness in
the $\hat t$ integration domain is not an issue, and we conclude that for $\xi,\eta$ in a
bounded region we conclude that $a_h$ is $h$ times a $\CI$ function.

We now consider the behavior of $|(\xi,\eta)|\to\infty$ to complete
showing that $a_h$ is a symbol.
The only subtlety in applying the stationary phase lemma is that the
domain of integration in $\hat t$ is not compact, so we need to
explicitly deal with the region $|\hat t|\geq 1$, say, assuming that the
amplitude is Schwartz in $\hat t$, uniformly in the other variables
(as it is in our case thanks to the exponential weight factor). Notice that as long
as the first derivatives of the phase in the integration variables
have a lower bound $c |(\xi,\eta)|\,|\hat t|^{-k}$ for some $k$, and for some $c>0$, the standard
integration by parts argument gives the rapid decay of the integral in
the large parameter $|(\xi,\eta)|$. At $h=0$ the phase
is
$\xi(\hat\lambda\hat t+\alpha\hat t^2)+\hat t\eta\cdot\omega$;
if $|\hat
t|\geq 1$, say, the $\hat\lambda$ derivative is $\xi\hat t$, which is
thus bounded below by $|\xi|$ in magnitude, so the only place where
one may not have rapid decay is at $\xi=0$ (meaning, in the spherical
variables, $\frac{\xi}{|(\xi,\eta)|}=0$). In this region one may use
$|\eta|$ as the large variable to simplify the notation slightly. The
phase is then with $\hat\xi=\frac{\xi}{|\eta|}$, $\hat\eta=\frac{\eta}{|\eta|}$,
$$
|\eta|(\hat\xi(\hat\lambda\hat t+\alpha\hat t^2)+\hat t\hat\eta\cdot\omega),
$$
with parameter differentials (ignoring the overall $|\eta|$ factor)
$$
\hat\xi\hat t\,d\hat\lambda,(\hat t\hat\eta+\hat
t^2\hat\xi\pa_\omega\alpha)\cdot
\,d\omega,(\hat\xi(\hat\lambda+2\alpha\hat
t)+\hat\eta\cdot\omega)\,d\hat t.
$$
With $\hat\Xi=\hat\xi\hat t$ and $\rho=\hat t^{-1}$ these are
$$
\hat\Xi\,d\hat\lambda,\hat t(\hat\eta+\hat\Xi\pa_\omega\alpha)\cdot
\,d\omega,(\hat\Xi(\rho\hat\lambda+2\alpha)+\hat\eta\cdot\omega)\,d\hat t,
$$
and now for critical points $\hat\Xi$ must vanish (as we already knew
from above), then the last of these gives that $\hat\eta\cdot\omega$
vanishes, but then the second gives that there cannot be a critical
point (in $|\hat t|\geq 1$). While this argument was at $h=0$, the full
phase derivatives
are
\begin{equation*}\begin{aligned}
&(\hat\xi\hat t(1+h^{1/2}\hat t\pa_\lambda\alpha+h\hat
t^2\pa_\lambda\Gamma^{(1)})+\hat\eta\cdot h\hat
t^2\pa_\lambda\Gamma^{(2)})\,d\hat\lambda,\\
&(\hat t\hat\eta+h^{1/2}\hat t^2\hat\eta\cdot\pa_\omega\Gamma^{(2)}+\hat
t^2\hat\xi\pa_\omega\alpha+h^{1/2}\hat t^3\hat\xi\pa_\omega\Gamma^{(1)})\cdot
\,d\omega,\\
&(\hat\xi(\hat\lambda+2\alpha\hat
t+3h^{1/2}\hat t^2\Gamma^{(1)}+h\hat
t^3\pa_t\Gamma^{(1)}) +\hat\eta\cdot\omega+2h^{1/2}\hat t\Gamma^{(2)}+h\hat
t^2\pa_t\Gamma^{(2)})\,d\hat t, 
\end{aligned}\end{equation*}
i.e.
\begin{equation*}\begin{aligned}
&(\hat\Xi(1+t\pa_\lambda\alpha+t^2\pa_\lambda\Gamma^{(1)})+\hat\eta\cdot t^2\pa_\lambda\Gamma^{(2)})\,d\hat\lambda, \\
&\hat t(\hat\eta+\hat\eta\cdot t\pa_\omega\Gamma^{(2)}+\hat\Xi\pa_\omega\alpha+t\hat\Xi\pa_\omega\Gamma^{(1)})\cdot
\,d\omega,\\
&(\hat\Xi(\hat\lambda\rho+2\alpha+3t\Gamma^{(1)}+
t^2\pa_t\Gamma^{(1)})+\hat\eta\cdot\omega+2t\Gamma^{(2)}+t^2\pa_t\Gamma^{(2)})\,d\hat t,
\end{aligned}\end{equation*}
and now all the additional terms are small if $T_0$ is small, where we assume
$|t|<T_0$, so the lack of critical points in the $h=0$ computation
implies the analogous statement (in $|\hat t|>1$) for the general
computation {\em assuming $T_0$ is sufficiently small}.

Now, if $t\in [T_0,T]$ is not so small, the same result can be achieved under a
no-conjugate points assumption, i.e.\ that the Jacobian
$\frac{\pa\gamma}{\pa(t,\lambda,\omega)}$ is full rank for $t$ away
from $0$. Notice that we might use different coordinate charts for
$z,z'$ as discussed at the end of the paragraph of
\eqref{eq:semicl-full-symbol}, with the factor $e^{-i(\xi
  x/h+\eta\cdot y/h^{1/2})}$ of the integrand irrelevant if one is to
prove rapid decay.
In
this case one can run the non-stationary phase argument directly for $t$ (as
opposed to $\hat t$) away from
$0$ and $(\lambda,\omega)$, showing that there are no stationary points, and thus, as
$|(\xi,\eta)|\to\infty$ or $h\to 0$, one can reduce to the case $\hat t=0$ discussed
below (as there are no other non-trivial contributions). Concretely, we need to keep in mind that the exponential weight
is bounded by $e^{-\ep/h}$ for some $\ep>0$ when $t$ is bounded away
from $0$ (and $h$ is sufficiently small), thus is rapidly decaying as
$h\to 0$; in particular this assures the smoothness of $a_h$, and its
rapid decay in $h$, for bounded
$(\xi,\eta)$. Now, there is $C_0>0$ such that if $|\xi|/h^{1/2}>C_0|\eta|$
then $\pa_t$ of the phase is bounded from below by a positive constant
multiple of $|\xi|/h$, since $\pa_t\gamma^{(1)}\neq 0$ by the
convexity properties of the foliation. Thus, in this case we deduce
rapid decay of the integral in $|\xi|/h$, hence also in
$|\eta|/h^{1/2}$. So let $\tilde\xi=\xi/h^{1/2}$, and assume that
$|\tilde\xi|<2C_0|\eta|$. Then the phase becomes
$$
h^{-1/2}(\tilde\xi(\gamma^{(1)}_{x,y,\lambda,\omega}(t)-x)+\eta\cdot (\gamma^{(2)}_{x,y,\lambda,\omega}(t)-y)),
$$
which is a standard $h^{1/2}$-semiclassical phase, whose
non-stationarity is assured by the no-conjugate points
hypothesis. Due to the exponential weight, the amplitude is rapidly
decreasing in $h$ regardless of the singular $\lambda$-dependence of
$\tilde\chi$, which makes stationary phase applicable, giving the
desired result of rapid decay in $|\eta|/h^{1/2}$, hence also in
$|\tilde\xi|/h^{1/2}=|\xi|/h^{1/2}$ in view of the constraint on $\tilde\xi$.

This discussion implies that we may work in $|\hat t|<2$, say, and one can use the standard {\em parameter-dependent} stationary phase lemma, see
e.g.\ \cite[Theorem~7.7.6]{Hor} for our stationary phase computation
in $(\hat t,\hat \lambda,\omega)$. At $h=0$, the
stationary points of the phase are $\hat t=0$,
$\xi\hat\lambda+\eta\cdot\omega=0$, which remain critical points for
$h$ non-zero due to the $h^{1/2}\hat t^2$ vanishing of the other terms, and
when $|\hat t|<2$ and $h$ is sufficiently small, so $h^{1/2}\hat t$ is small, there are no other critical
points. (One can see this in a different way: above we worked with
$|\hat t|\geq 1$, but for any $\ep>0$, $|\hat t|\geq\ep$ would have
worked equally.) These critical points lie on a smooth codimension 2
submanifold of the parameter space. For the following argument it is
useful to consider $(\xi,\eta)$ jointly, and write
$\hat\xi=\frac{\xi}{|(\xi,\eta)|}$,
$\hat\eta=\frac{\eta}{|(\xi,\eta)|}$. Moreover, we write
$\theta=(\hat\lambda,\omega)$, and decompose it into a parallel and
orthogonal component relative to $(\hat\xi,\hat\eta)$:
$\theta^\parallel=(\hat\xi,\hat\eta)\cdot(\hat\lambda,\omega)$, resp.\
$\theta^\perp$. At $h=0$, the $(\hat t,\theta^\parallel)$-Hessian matrix of the phase
$$
|(\xi,\eta)|(\hat\xi(\hat\lambda\hat t+\alpha\hat t^2)+\hat
t\hat\eta\cdot\omega)=|(\xi,\eta)|(\hat t\hat\theta^\parallel+\hat\xi\alpha\hat t^2)
$$
within
$\hat t=0$ (where the critical set lies) is
$$
|(\xi,\eta)|\begin{pmatrix}
  2\hat\xi\alpha&1\\1&0\end{pmatrix},
$$
which is always invertible. This also implies, by continuity (and
homogeneity in $(\xi,\eta)$) of the Hessian the invertibility for
small $h$.
We thus use the stationary phase lemma in the
$(\hat t,\theta^\parallel)$ variables, which shows that $a_h$ is ($h$
times, due to the integration variable change, already mentioned for
the finite $\xi,\eta$ discussion!) a
symbol of order $-1$, since the stationary phase is with respect to a 2-dimensional
space.
\end{proof}

We now proceed to compute the semiclassical principal symbol:

\begin{prop}
There exists $\tilde\chi$ of compact support such that the operator
$A_h\in h\Psihh^{-1}(\tilde M;\cF)$ is elliptic on $M$.
  \end{prop}

\begin{proof} For the semiclassical principal symbol computation  
we may simply set $h=0$ in the above rescaled expression used
for the stationary phase argument.
Thus,  with $\tilde\chi=\chi(\lambda/h^{1/2})=\chi(\hat\lambda)$,
we have that
\begin{equation}\begin{aligned}\label {eq:aj-1-form-pre-xi}
a_h(x,y,\xi,\eta)
=
h\int 
&e^{i(\xi h^{-1}(\gamma^{(1)}_{x,y,x\hat\lambda,\omega}(h^{1/2}\hat
  t)-x)+\eta h^{-1/2} x^{-1}(
  \gamma^{(2)}_{x,y,x\hat\lambda,\omega}(h^{1/2}\hat t)-y))}\\
&\qquad e^{-(\hat\lambda\hat t+\alpha\hat
t^2)}\chi(\hat\lambda)
\,d\hat t\,d\hat \lambda\,d\omega\\
=h\int &e^{i(\xi(\hat\lambda\hat t+\alpha\hat t^2+h^{1/2}\hat
t^3\Gamma^{(1)}(x,y,h^{1/2}\hat\lambda,\omega,h^{1/2}\hat t))+\eta\cdot(\omega\hat
t+h^{1/2}\hat t^2\Gamma^{(2)}(x,y,h^{1/2}\hat\lambda,\omega,h^{1/2}\hat t)))}\\
&\qquad e^{-(\hat\lambda\hat t+\alpha\hat
t^2)}\chi(\hat\lambda)\,d\hat t\,d\hat \lambda\,d\omega,
\end{aligned}\end{equation}
up to errors that are
$O(h^{1/2}\langle\xi,\eta\rangle^{-1})$ relative to
the a priori order, $-1$, arising from the $0$th order symbol in
the oscillatory integral and the 2-dimensional space in which the
stationary phase lemma is applied, and semiclassical order $-1$,
corresponding to the factor of $h$ in front. Factoring out the overall
$h$, this in turn becomes, modulo
$O(h^{1/2})$
errors, i.e.\ at the semiclassical foliation principal symbol level,
\begin{equation}\begin{aligned}
    a_0(x,y,\xi,\eta)=\int &e^{i(\xi(\hat\lambda\hat t+\alpha\hat t^2)+\eta\cdot\omega\hat
t)} e^{-(\hat\lambda\hat t+\alpha\hat
t^2)}\chi(\hat\lambda)\,d\hat t\,d\hat \lambda\,d\omega.
\end{aligned}\end{equation}

Now computing the principal symbol of this, i.e.\ the behavior as
$|(\xi,\eta)|\to\infty$, we consider
the critical points of the phase, $\hat t=0$,
$\theta^\parallel\equiv\hat\xi\hat\lambda+\hat\eta\cdot\omega=0$, where $\theta^\perp$ is the
variable along the critical set (where $\hat t$ and $\theta^\parallel$
vanish), which gives, up to an
overall elliptic factor, keeping in mind that $\hat\lambda$ depends on
$\theta^\perp$ along this equator (namely along $\theta^\parallel=0$),
$$
\int_{\sphere^{n-2}}\chi(\hat\lambda(\theta^\perp)) \,d\theta^\perp,
$$
which is elliptic for $\chi\geq 0$ with $\chi(0)>0$ since the
codimension one planes (intersected with a sphere) $\theta^\parallel=0$ and
$\hat\lambda=0$ necessarily intersect in at least a line (intersected
with the sphere) as the dimension is $n\geq 2+1=3$.

This of course implies that for finite, but sufficiently large,
$(\xi,\eta)$, the semiclassical symbol $a_0$ is elliptic. For general
finite $(\xi,\eta)$ it is harder to compute $a_0$ explicitly for
general $\chi$. However, when $\chi$ is a Gaussian, the computation is
straightforward. We write
\begin{equation*}\begin{aligned}
    a_0(x,y,\xi,\eta)=\int &e^{-(\alpha(1-i\xi)\hat t^2+\hat
      t(\hat\lambda(1-i\xi)-i\eta\cdot\omega)}
    \chi(\hat\lambda)\,d\hat t\,d\hat \lambda\,d\omega\\
    =\int &e^{-\alpha(1-i\xi)(\hat
      t+\frac{\hat\lambda(1-i\xi)-i\eta\cdot\omega}{2\alpha(1-i\xi)})^2}e^{\frac{(\hat\lambda(1-i\xi)-i\eta\cdot\omega)^2}{4\alpha(1-i\xi)}}
    \chi(\hat\lambda)\,d\hat t\,d\hat \lambda\,d\omega\\
    =c\int &\alpha^{-1/2}(1-i\xi)^{-1/2}e^{\frac{(\hat\lambda(1-i\xi)-i\eta\cdot\omega)^2}{4\alpha(1-i\xi)}}
    \chi(\hat\lambda)\,d\hat \lambda\,d\omega\\
    =c\int&\alpha^{-1/2}(1-i\xi)^{-1/2}e^{\frac{\hat\lambda^2(1-i\xi)}{4\alpha}-i\frac{\hat\lambda}{2\alpha}\eta\cdot\omega-\frac{(\eta\cdot\omega)^2}{4\alpha(1-i\xi)}}\chi(\hat\lambda)\,d\hat \lambda\,d\omega,
  \end{aligned}\end{equation*}
with $c$ a non-zero constant.
Now a particularly helpful choice is
$\chi(\hat\lambda)=e^{-\hat\lambda^2/(2\alpha)}$, for then we have
\begin{equation*}\begin{aligned}
    a_0(x,y,\xi,\eta)=c\int&\alpha^{-1/2}(1-i\xi)^{-1/2}e^{-\frac{\hat\lambda^2(1+i\xi)}{4\alpha}-i\frac{\hat\lambda}{2\alpha}\eta\cdot\omega-\frac{(\eta\cdot\omega)^2}{4\alpha(1-i\xi)}}\,d\hat
    \lambda\,d\omega\\
    =c\int&\alpha^{-1/2}(1-i\xi)^{-1/2}e^{-\frac{1+i\xi}{4\alpha}(\hat\lambda+i\frac{\eta\cdot\omega}{1+i\xi})^2-\frac{(\eta\cdot\omega)^2}{4\alpha(1+i\xi)}-\frac{(\eta\cdot\omega)^2}{4\alpha(1-i\xi)}}\,d\hat 
    \lambda\,d\omega \\
        =c'\int&(1+\xi^2)^{-1}e^{-\frac{(\eta\cdot\omega)^2}{2\alpha(1+\xi^2)}}\,d\omega 
  \end{aligned}\end{equation*}
and the integral is now positive since the integrand is such, while
$c'$ is a new non-zero constant. It
follows immediately that the same positivity property is maintained if
$\chi$ is close to the Gaussian in the space of Schwartz functions,
which can be achieved by taking a compactly supported $\chi$.
\end{proof}

In view
of the errors of the elliptic parametrix construction being small in the semiclassical Sobolev spaces, for
sufficiently small $h$ (but $h$ can be fixed to such a small value, so
$A_h$ is a standard pseudodifferential operator then, and the Sobolev
spaces are standard Sobolev spaces with an equivalent norm!), as
discussed in Section~\ref{sec:semicl}, {\em this proves Theorem~\ref{thm:main}.}

The scattering version is quite similar; recall that $x=\xt+c$ in
terms of the original foliation function $\xt$. The cutoff scaling we
use in this case is
\begin{equation}\label{eq:tilde-chi-sc}
\tilde\chi(z,\lambda/(xh^{1/2}),\omega).
\end{equation}
Thus,
writing scattering covectors
as $\xisc\,\frac{dx}{x^2}+\etasc\,\frac{dy}{x}$, i.e.\ substituting
$$
\xisc=x^2\xi,\ \etasc=x\eta,
$$
into \eqref{eq:semicl-full-symbol}
\begin{equation}\begin{aligned}\label{eq:semicl-sc-full-symbol}
a_{h}(x,y,\xisc,\etasc)=
\int e^{-\Phi(x)/h} &e^{\Phi(x(\gamma_{z,\lambda,\omega}(t)))/h}\tilde\chi(z,\lambda/(xh^{1/2}),\omega)\\
&\qquad 
e^{ix^{-2}\xisc\cdot(\gamma^{(1)}_{z,\lambda,\omega}(t)-x)/h}e^{ix^{-1}\etasc\cdot(\gamma^{(2)}_{z,\lambda,\omega}(t)-y)/h^{1/2}}\,dt\,|d\nu|,
\end{aligned}\end{equation}
with $\Phi(x)=x^{-1}$ in this case.

\begin{prop}
  Let $\tilde M_c=\tilde M\cap \{\tilde x\geq -c\}=\tilde M\cap\{x\geq
  0\}$. Then
  $A_h\in h\Psischh^{-1,-2}(\tilde M_c;\cF)$.
\end{prop}

\begin{proof}
  We change the variables of integration to $\hat t=t/(h^{1/2} x)$, and
$\hat\lambda=\lambda/(h^{1/2}x)$, so again the $\hat\lambda$ integral is in fact over
a fixed compact interval, but the $\hat t$ one is over $|\hat t|<T/(x\sqrt{h})$
which grows as $h\to 0$ or $x\to 0$.
We get that the
phase is
\begin{equation*}\begin{aligned}
&\xisc(\hat\lambda\hat t+\alpha\hat t^2+xh^{1/2}\hat
t^3\Gamma^{(1)}(x,y,xh^{1/2}\hat\lambda,\omega,xh^{1/2}\hat t))\\
&\qquad+\etasc\cdot(\omega\hat t+xh^{1/2}\hat t^2\Gamma^{(2)}(x,y,h^{1/2}\hat\lambda,\omega,xh^{1/2}\hat t)),
\end{aligned}\end{equation*}
while the exponential damping factor (which we regard as a Schwartz
function, part of the amplitude, when one regards $\hat t$ as a
variable on $\RR$) is
\begin{equation*}\begin{aligned}
&-1/(hx)+1/(h\gamma^{(1)}_{x,y,\lambda,\omega}(t))\\
&=
-h^{-1}(\lambda t+\alpha
t^2+t^3\Gamma^{(1)}(x,y,\lambda,\omega,t))x^{-1}(x+\lambda t+\alpha
t^2+t^3\Gamma^{(1)}(x,y,\lambda,\omega,t))^{-1}\\
&=-(\hat\lambda\hat t+\alpha\hat
t^2+\hat t^3 xh^{1/2}\hat\Gamma^{(1)}(x,y,xh^{1/2}\hat\lambda,\omega,xh^{1/2}\hat t)),
\end{aligned}\end{equation*}
with $\hat\Gamma^{(1)}$ a smooth function. Thus, for $\xi,\eta$ in a
bounded region we conclude that $a_h$ is a $\CI$
function. Furthermore, we observe that with $(\xisc,\etasc)$ in place
of $(\xi,\eta)$, and in the new integration variables $\hat t$ and
$\hat\lambda$, \eqref{eq:semicl-sc-full-symbol} has the same form as
\eqref{eq:semicl-full-symbol}, so identical stationary phase arguments
are applicable. In particular, the $t$ bounded away from $0$ case
proceeds analogously with $xh^{1/2}$ playing the role of $h^{1/2}$, keeping
in mind that the exponential weight is bounded by $e^{-\ep/(xh)}$ for
$t$ bounded away from $0$, so
is rapidly decaying in $xh$; in this case
$|\xisc|/(xh^{1/2})>C_0|\etasc|$ assures the possibility of
$t$-integration by parts to obtain rapid decay in $xh$, while if
$|\widetilde\xisc|=|\xisc|/(xh^{1/2})<2C_0|\etasc|$, then one can apply a
standard no-stationary point argument as above under the no-conjugate
point assumption since the phase is $x^{-1}h^{-1/2}$ times a usual
homogeneous degree 1 phase in $(\widetilde\xisc,\etasc)$, giving rapid
decay in $x^{-1}h^{-1/2}|\etasc|$, thus in $x^{-2}h^{-1}|\xisc|$ as
well.
\end{proof}

Finally, it remains to compute the semiclassical foliation scattering
principal symbol, which is, taking into account the density factor
$hx^2$ from the change of variables,
\begin{equation}\begin{aligned}\label{eq:finite-points-sc-symbol}
    a_0(x,y,\xisc,\etasc)=hx^2\int &e^{i(\xisc(\hat\lambda\hat t+\alpha\hat t^2)+\etasc\cdot\omega\hat
t)} e^{-(\hat\lambda\hat t+\alpha\hat
t^2)}\chi(\hat\lambda)\,d\hat t\,d\hat \lambda\,d\omega.
\end{aligned}\end{equation}
Then completely analogously to the above computation yields that as $|(\xisc,\etasc)|\to\infty$, up to an
overall elliptic factor, we have
$$
\int_{\sphere^{n-2}}\chi(\hat\lambda(\theta^\perp)) \,d\theta^\perp,
$$
which is elliptic for $\chi\geq 0$ with $\chi(0)>0$. Further, the
ellipticity at finite points follows the same computation as above,
for the same choice of $\chi$, $\chi(\hat\lambda)=e^{-\hat\lambda^2/(2\alpha)}$,
with $(\xi,\eta)$ replaced by $(\xisc,\etasc)$.  Again, in view
of the errors of the elliptic parametrix construction being small in the semiclassical Sobolev spaces as
discussed in Section~\ref{sec:semicl}, {\em this proves Theorem~\ref{thm:main-sc}.}

\bibliographystyle{plain}
\bibliography{sm}

\def\cprime{$'$} \def\cprime{$'$}
\begin{thebibliography}{10}

\bibitem{DUV:Diffraction}
Maarten de~Hoop, Gunther Uhlmann, and Andr{\'a}s Vasy.
\newblock Diffraction from conormal singularities.
\newblock {\em Ann. Sci. \'Ec. Norm. Sup\'er. (4)}, 48(2):351--408, 2015.

\bibitem{Gannot-Wunsch:Conormal}
Oran Gannot and Jared Wunsch.
\newblock Semiclassical diffraction by conormal singularities.
\newblock {\em Preprint, arXiv:1806.01813}, 2018.

\bibitem{Hor}
L.~H\"ormander.
\newblock {\em The analysis of linear partial differential operators, {\rm vol.
  1-4}}.
\newblock Springer-Verlag, 1983.

\bibitem{Melrose:Microlocal-notes}
R.~B. Melrose.
\newblock Lecture notes for `18.157: {I}ntroduction to microlocal analysis'.
\newblock Available at
  \texttt{http://math.mit.edu/~rbm/18.157-F09/18.157-F09.html}, 2009.

\bibitem{Melrose:Transformation}
Richard~B. Melrose.
\newblock Transformation of boundary problems.
\newblock {\em Acta Math.}, 147(3-4):149--236, 1981.

\bibitem{Melrose:Atiyah}
Richard~B. Melrose.
\newblock {\em The {A}tiyah-{P}atodi-{S}inger index theorem}, volume~4 of {\em
  Research Notes in Mathematics}.
\newblock A K Peters Ltd., Wellesley, MA, 1993.

\bibitem{RBMSpec}
Richard~B. Melrose.
\newblock Spectral and scattering theory for the {L}aplacian on asymptotically
  {E}uclidian spaces.
\newblock In {\em Spectral and scattering theory ({S}anda, 1992)}, volume 161
  of {\em Lecture Notes in Pure and Appl. Math.}, pages 85--130. Dekker, New
  York, 1994.

\bibitem{Parenti:Operatori}
Cesare Parenti.
\newblock Operatori pseudo-differenziali in {$R^{n}$} e applicazioni.
\newblock {\em Ann. Mat. Pura Appl. (4)}, 93:359--389, 1972.

\bibitem{Sjostrand-Zworski:Fractal}
Johannes Sj{\"o}strand and Maciej Zworski.
\newblock Fractal upper bounds on the density of semiclassical resonances.
\newblock {\em Duke Math. J.}, 137(3):381--459, 2007.

\bibitem{Stefanov-Uhlmann:Rigidity}
Plamen Stefanov and Gunther Uhlmann.
\newblock Rigidity for metrics with the same lengths of geodesics.
\newblock {\em Math. Res. Lett.}, 5(1-2):83--96, 1998.

\bibitem{Stefanov-Uhlmann:Integral}
Plamen Stefanov and Gunther Uhlmann.
\newblock Integral geometry of tensor fields for a class of non-simple
  riemannian manifolds.
\newblock {\em Amer. J. Math.}, 130:239--268, 2008.

\bibitem{Stefanov-Uhlmann-Vasy:Rigidity-Normal}
Plamen Stefanov, Gunther Uhlmann, and Andras Vasy.
\newblock Local and global boundary rigidity and the geodesic {X}-ray transform
  in the normal gauge.
\newblock {\em Preprint, arXiv:1702.03638}, 2017.

\bibitem{Shubin:Pseudodifferential}
M.~A. {\v{S}}ubin.
\newblock Pseudodifferential operators in {$R^{n}$}.
\newblock {\em Dokl. Akad. Nauk SSSR}, 196:316--319, 1971.

\bibitem{Uhlmann-Vasy:X-ray}
Gunther Uhlmann and Andr{\'a}s Vasy.
\newblock The inverse problem for the local geodesic ray transform.
\newblock {\em Invent. Math.}, 205(1):83--120, 2016.

\bibitem{Vasy-Zworski:Semiclassical}
A.~Vasy and M.~Zworski.
\newblock Semiclassical estimates in asymptotically {E}uclidean scattering.
\newblock {\em Commun. Math. Phys.}, 212:205--217, 2000.

\bibitem{Vasy:Minicourse}
Andr\'as Vasy.
\newblock A minicourse on microlocal analysis for wave propagation.
\newblock In {\em Asymptotic analysis in general relativity}, volume 443 of
  {\em London Math. Soc. Lecture Note Ser.}, pages 219--374. Cambridge Univ.
  Press, Cambridge, 2018.

\bibitem{Zachos:Thesis}
Evangelie Zachos.
\newblock {\em The {X}-ray transform on asymptotically {E}uclidean spaces}.
\newblock PhD thesis, Stanford University, 2020.

\bibitem{Zworski:Semiclassical}
Maciej Zworski.
\newblock {\em Semiclassical analysis}, volume 138 of {\em Graduate Studies in
  Mathematics}.
\newblock American Mathematical Society, Providence, RI, 2012.

\end{thebibliography}
\end{document}